\documentclass[11pt]{amsart}
\usepackage[T1]{fontenc} 
\usepackage{amsmath,amsthm,amsfonts,amssymb,bm,graphicx,xcolor,mathtools, a4wide}

\usepackage
{hyperref}
\hypersetup{colorlinks=true,citecolor=blue,linkcolor=blue,urlcolor=blue,
pdfstartview=FitH }

\usepackage{parskip}
\numberwithin{equation}{section}

\theoremstyle{plain}
\newtheorem{theorem}{Theorem}
\newtheorem{cor}[theorem]{Corollary}
\newtheorem{lemma}{Lemma}

\mathtoolsset{showonlyrefs=true}
\theoremstyle{definition}

\newtheorem{remark}{Remark}

\renewcommand{\geq}{\geqslant}
\renewcommand{\leq}{\le}

\newcommand{\Z}{\mathbb{Z}}

\newcommand{\R}{\mathbb{R}}
\newcommand{\C}{\mathbb{C}}

\newcommand{\p}{\mathbf{p}}

\newcommand{\re}{\operatorname{Re}}

\newcommand{\tr}{\operatorname{tr}}

\newcommand{\SO}{\operatorname{SO}}
\newcommand{\Sp}{\operatorname{Sp}}

\newcommand{\bdd}{\begin{center}\begin{tikzcd}}
\newcommand{\bd}{\begin{tikzcd}}
\newcommand{\edd}{\end{tikzcd}\end{center}}
\newcommand{\ed}{\end{tikzcd}}
\newcommand{\bdp}{\begin{center}\begin{tikzpicture}}
\newcommand{\edp}{\end{tikzpicture}\end{center}}
\newcommand{\bi}{\begin{itemize}}
\newcommand{\ei}{\end{itemize}}
\newcommand{\bt}{\begin{tikzpicture}}
\newcommand{\et}{\end{tikzpicture}}
\newcommand{\ba}{\[\begin{aligned}}
\newcommand{\ea}{\end{aligned}\]}
\newcommand{\bp}{\begin{pmatrix}}
\newcommand{\ep}{\end{pmatrix}}
\newcommand{\bv}{\begin{vmatrix}}
\newcommand{\ev}{\end{vmatrix}}
\newcommand{\bb}{\begin{bmatrix}}
\newcommand{\eb}{\end{bmatrix}}
\newcommand{\bB}{\begin{Bmatrix}}
\newcommand{\eB}{\end{Bmatrix}}
\newcommand{\bea}{\begin{enumerate}[label=(\alph*)]}
\newcommand{\ber}{\begin{enumerate}[label=(\roman*)]}
\newcommand{\ben}{\begin{enumerate}[label=(\arabic*)]}
\newcommand{\ee}{\end{enumerate}}

\numberwithin{equation}{section}

\makeatletter
\def\Ddots{\mathinner{\mkern1mu\raise\p@
\vbox{\kern7\p@\hbox{.}}\mkern2mu
\raise4\p@\hbox{.}\mkern2mu\raise7\p@\hbox{.}\mkern1mu}}
\makeatother
\makeatletter
\DeclareRobustCommand\widecheck[1]{{\mathpalette\@widecheck{#1}}}
\def\@widecheck#1#2{%
    \setbox\z@\hbox{\m@th$#1#2$}%
    \setbox\tw@\hbox{\m@th$#1%
       \widehat{%
          \vrule\@width\z@\@height\ht\z@
          \vrule\@height\z@\@width\wd\z@}$}%
    \dp\tw@-\ht\z@
    \@tempdima\ht\z@ \advance\@tempdima2\ht\tw@ \divide\@tempdima\thr@@
    \setbox\tw@\hbox{%
       \raise\@tempdima\hbox{\scalebox{1}[-1]{\lower\@tempdima\box
\tw@}}}%
    {\ooalign{\box\tw@ \cr \box\z@}}}
\makeatother

\begin{document}

\author{Valentin Blomer}
\address{Mathematisches Institut, Endenicher Allee 60, 53115 Bonn, Germany}
\email{blomer@math.uni-bonn.de}
 
\author{Soumya Das}
\address{Department of Mathematics, Indian Institute of Science, Bangalore - 560012, India}
\email{soumya@iisc.ac.in}

 \title{Moments and joint nonvanishing of symplectic $L$-functions}


\keywords{Standard $L$-functions, spectral summation formula, Siegel modular forms, moments, non-vanishing}

\begin{abstract} We compute an asymptotic formula for a moment involving the spinor and the standard $L$-functions for holomorphic Siegel cusp forms of degree two and large weight $k$. Applications include simultaneous non-vanishing statements and lower bounds for second moments. 
\end{abstract}

\subjclass[2020]{11F46, 11F66, 11F72}

\setcounter{tocdepth}{2}  \maketitle 

\maketitle

\section{Introduction}
 
 \subsection{Statement of results}
 Let $F$ be a holomorphic Siegel cusp form of weight $k$ and genus $2$ for the group $ \Sp(4,\Z) $, which is an eigenfunction of the Hecke algebra.  Attached to $F$ are two basic $L$-functions:
 \begin{itemize}
 \item the spinor $L$-function $L(s, F)$ of degree four and conductor $k^2$;
 \item the standard $L$-function $D(s,F)$ of degree five and conductor $k^4$.
 \end{itemize}
 
  This paper contributes to the asymptotic evaluation of moments of these $L$-functions by means of a suitable spectral summation formula in a higher rank situation. 
  For the symplectic group $\Sp(4,\Z)$, the relevant formula to average (abelian) Fourier coefficients over the space of cusp  forms of weight $k$ was worked out by Kitaoka \cite{Ki} as an analogue of the well-known Petersson formula for elliptic modular forms. It is a challenge and interesting in its own right to analyze the (matrix) Kloosterman sums and (matrix) Bessel functions appearing on the arithmetic side of this formula. 
  A first template for this kind of analysis in the context of spinor $L$-functions was given in \cite{Bl}. 
  
  Not only because of its conductor and degree, but also because of its Dirichlet series expansion (see \eqref{DF}), the standard $L$-function is in many respects the more complicated of the two $L$-functions, for which to the best of our knowledge no moment results exist. 
  Our basic motivation is a first moment of $D(1/2, F)$ over an eigenbasis $\mathcal{B}^{(2)}_k$ of the space of weight $k$ Siegel cusp forms, which has cardinality of size $
\asymp k^3$. Since we apply Kitaoka's version of the Petersson formula \cite{Ki}, this comes naturally with ``harmonic weights''
$$\omega_F := \frac{\pi^{1/2}}{4} (4\pi)^{3 - 2k} \Gamma(k-3/2)\Gamma(k-2) \frac{a_F(I)^2}{\| F \|^2}. $$
Here, as usual, we denote the Fourier coefficients of $F$ by $a_F(M)$ for $M \in \mathcal{S}$, the set of half-integral symmetric $2$-by-$2$ matrix with integral diagonal (see Section \ref{sec21} for the normalization of $a_F(M)$).
These weights satisfy (\cite[(1.8)]{Bl})
\[
\sum_{F \in \mathcal{B}_k^{(2)}} \omega_F = 1 + O(e^{-k}).
\]
By B\"ocherer's conjecture, proved by Furusawa and Morimoto \cite{FS}, they are themselves special values of $L$-functions: if $F$ is not a Saito-Kurokawa lift, then (cf.\ \cite[Theorem 1.2]{DPSS})

\begin{equation}\label{dpssth12}
 \omega_F = \frac{128\pi^5 \Gamma(2k-4)}{\Gamma(2k-1)} \frac{L(1/2, F)  L(1/2, F \times \chi_{-4} )    }{L(1, \pi_F, \text{Ad})}.
 \end{equation}
 This  naturally leads to a mixed moment involving both standard and spinor $L$-functions. We have denoted by  $\chi_{-4}$   the non-trivial Dirichlet character modulo $4$. Analogous results as Theorem \ref{thm1} exist for any other (fixed) primitive  quadratic character. 
 
 \begin{theorem}\label{thm1} We have
 \[
 \mathcal{M}_k := \sum_{F \in \mathcal{B}_k^{(2)}} \omega_F D(1/2,F) = 2 L(3/2, \chi_{-4}) \log k  + c_0 + O(k^{-\eta})
\]
for any $\eta < 1/4$ with
\[c_0 = L(3/2, \chi_{-4}) \Big(\frac{3}{2} \gamma - \frac{\pi}{4} -\frac{7}{2} \log 2 - \frac{5}{2} \log \pi\Big) + L'(3/2, \chi_{-4}).
\]
The contribution of the lifts is negligible: if $\tilde{\mathcal{B}}_k^{(2)}$ denotes an eigen-basis of non-lifts, then
\begin{equation}\label{thm1eq}
\frac{16\pi^5}{k^3}\sum_{F \in \tilde{\mathcal{B}}_k^{(2)}} \frac{ L(1/2, F)  L(1/2, F \times \chi_{-4} )    D(1/2, F) }{L(1, \pi_F,\text{{\rm Ad}})}  = 2 L_{\chi_{-4}}(3/2) \log k  + c_0 + O(k^{-\eta}).
\end{equation}
\end{theorem}

The left hand side of \eqref{thm1eq} features   an $L$-function of degree $13$.  As a point of reference, the combined conductor of the left hand side is $k^8$, relative to a family of size $k^3$. This is comparable to \cite{Bl}, but with the more complicated $L$-function $D(\cdot, F)$. In the genus one case this  corresponds in terms of its Dirichlet series expansion to the notoriously difficult symmetric square $L$-functions. Indeed, while the spinor $L$-function  involves only the radial Fourier coefficients $a_F(mI)$, $m \in \mathbb{N}$, the standard $L$-function is built from the Fourier coefficients $a_F(MM^{\top})$ for $M \in {\rm SL}_2(\mathbb{Z}) \backslash M_2^+(\mathbb{Z})$.

Using the asymptotic formula (see e.g.\ \cite[Theorem~5]{Ig}) 
\[ D := \#\mathcal{B}_k^{(2)} = \frac{1}{8640}k^3 + O(k^2) \text{ and } V := \text{vol}({\rm Sp}_4( \mathbb{Z})\backslash \mathbb{H}_2) = \frac{\pi^3}{270},
\]
then based on \eqref{dpssth12} we can write the formula more intrinsically as
 \[ \frac{1}{V D}   \sum_{F \in \tilde{\mathcal{B}}_k^{(2)}} \frac{L(1/2, F)L(1/2, F \times \chi_{-4})D(1/2, F)}{L(1, \pi_F,\text{{\rm Ad}})}  = \frac{2}{3} \frac{L(3/2, \chi_{-4})}{\zeta(2)} \log k  + \frac{c_0}{3\zeta(2)} + O(k^{-\eta}).
\]
From Theorem \ref{thm1} we immediately conclude  the following simultaneous non-vanishing statement. 
\begin{cor}\label{Kor} If $k$ is sufficiently large, there exists $F \in \tilde{\mathcal{B}}_k^{(2)}$ such that
$$L_F(1/2), \quad  L(1/2, F \times \chi_{-4}), \quad D(1/2, F)$$
are simultaneously non-zero and moreover $D(1/2, F)>0$.
\end{cor}

The non-negativity of $D(1/2, F)$ is a consequence of GRH and the intermediate value theorem (since $D(s, F) > 0$ for $\Re s > 1$), but unconditionally does not follow from a known period formula. 


\begin{cor}\label{Kor3}
     If $k$ is sufficiently large, we have the lower bound
\begin{align}
\sum_{F \in \tilde{\mathcal{B}}_k^{(2)}} \frac{|D(1/2, F)|^2}{ L(1, \pi_F,\text{{\rm Ad}})} \gg k^3.
\end{align}
\end{cor}
This follows immediately from the Cauchy-Schwarz inequality, if we invoke the second moment result from \cite[Theorem~2]{Bl}
specialized to $q_1=1, q_2=-4$. Since the family is symplectic \cite[Section 10.2]{KWY}, the expected order of magnitude for the second moment is $k^3 (\log k)^3$. Combining the Cauchy-Schwarz inequality with a mollification argument (note that both Theorem \ref{thm1} and \cite{Bl} come with a power saving error term), it might be possible to upgrade Corollary \ref{Kor3} accordingly.

\subsection{Sketch of the proof}
We briefly  describe the strategy of the proof. We insert an approximate functional equation \eqref{approx} for $D(s, F)$ into Kitaoka's formula \eqref{Kita}. The diagonal term provides the main term and we need to find power saving estimates for the rank one and the rank two term involving symplectic Kloosterman sums and Bessel functions. 

Typically (e.g.\ in \cite{Bl}) the rank one term is not particularly difficult, but in our present situation the   trivial bound does not suffice, even assuming best possible cancellation in two-dimensional exponential sums. 
Instead, an additional  application of Poisson summation is necessary; see Section \ref{rank1}. In Lemma~\ref{exp}, we essentially prove a square-root cancellation in the above-mentioned exponential sums, which might be of independent interest and is also used to treat the rank one terms. 

In the rank two term, the first key step is to glue together the matrix summation variable from the approximate functional equation and the matrix modulus of the Kloosterman sum. Then we apply Poisson summation in a suitable set of variables -- in some ranges all entries of the matrices, in some ranges only two variables in a carefully chosen coordinate system. This is coupled with a detailed analysis of the matrix-Bessel function, which depends on the relative sizes of the singular values of the involved matrices. In particular, the matrix-Bessel function has less oscillation for orthogonal matrices, which affects the analysis in a crucial way. We therefore need to quantify how ``close'' the matrices in question are to being orthogonal. This argument is carried out in Section \ref{rank2} and is, not surprisingly, the most difficult part of the proof. 

\subsection{Further results and remarks}

We refer the reader to \cite{Zh, KWY} for some analysis of standard $L$-functions in the context of low-lying zeros. We also mention the recent work \cite{wwyy} for multiplicity one statements and moments of \emph{fixed} symplectic $L$-functions by Dirichlet characters.  

It would also be interesting to study the standard $L$-function by means of the pullback formula (see \cite{Bo}):
\begin{displaymath}
\begin{split}
F(Z) D(s, F) = (-1)^k & 2^{1-k}   \pi^{6-3s}  \frac{\Gamma(2s + k - 1) \Gamma(s-1)\Gamma(s - 3/2)}{\Gamma(s - 1/2)}\zeta(2s) \zeta(4s-2) \\
& \times \Big\langle  F, E_k\Big(\Big(\begin{matrix} -\bar{Z} & \\ & * \end{matrix}\Big) ; \frac{\bar{s}}{2} - \frac{k}{2}+ 1 \Big)\Big\rangle
\end{split}
\end{displaymath}
where $\langle .,. \rangle$ denotes the weight $k$ Petersson inner product and 
\[  E_k(Z; s) = \sum_{(\begin{smallmatrix} * & *\\ C & D\end{smallmatrix}) \in \Gamma_{\infty} \backslash {\rm Sp}_4(\mathbb{Z})} \det (CZ + D)^{-k} |\det (CZ+D)|^{-2s} (\det Y)^s \]
is the real-analytic Siegel Eisenstein series of even weight $k$. 


\subsection{Acknowledgements}
The   first author was supported by DFG through SFB-TRR 358 and EXC-2047/1 - 390685813 and by ERC Advanced Grant 101054336. This work started when S.D. was enjoying the hospitality at the University of Bonn in 2022 and was supported by an Alexander von Humboldt Fellowship. It is a pleasure for him to thank the University of Bonn for providing an excellent working atmosphere.
S.D. is supported by an ANRF grant ANRF/ARGM/2025/000535/MTR from the Ministry of Science and Technology, India, and thanks IISc. Bangalore, UGC Centre for Advanced Studies, DST, India, for financial support.

\section{Notation and Preliminaries}

\subsection{General notation}
We used mostly standard notation throughout the paper. For a matrix $M$, its transpose is denoted as $M^\top$. We denote by $M_2(R)$ the set of $2 \times 2$ matrices over a commutative ring $R$. The norm $\| M \|$ of a matrix is always the Frobenius norm $\| M \| = \text{tr}(MM^{\top})^{1/2}.$ If $M$ is invertible, we write $M^{\text{adj}} = M^{-1} \det M$.  

The notations $e(z):= \exp(2 \pi i z)$ and $\Re z$ for the real part of $z \in \mathbb{C}$ are used extensively. We write $\mathbb{N}_0$ to denote the set of all non-negative integers.

We denote the space of holomorphic Siegel cusp forms of genus (or degree) $2$ and level one by $S^{(2)}_k$. The orthogonal set of a Hecke eigenbasis of $S^{(2)}_k$ is denoted by $\mathcal{B}_k^{(2)}$, and the subset $\mathcal{\tilde B}_k^{(2)} \subset \mathcal{B}_k^{(2)}$ consists of the set of those eigenforms which are not Saito-Kurokawa lifts from degree $1$. For a thorough reference on the Saito-Kurokawa lifts, we refer the reader to \cite{FPRS}; see also \cite[Section~7.3]{DK}.

As is common, $\varepsilon$ will denote an arbitrarily small positive number, which may change from line to line. We will require a more refined notion of this convention, and this is explained at the beginning of Section~\ref{diag-section} and \eqref{layered} below. We call a quantity negligible if it is $\ll_A k^{-A}$ for any $A > 0$. 

\subsection{The standard \texorpdfstring{$L$}{L}-function}\label{sec21}

Let us briefly recall  some facts about the standard $L$-function. 

Let $F\in \mathcal{B}_k^2$ be an eigenform for all Hecke operators $T(n)$. 
Let $\{\lambda_F(n)\}$ (normalized so that the functional equations relate $s \to 1-s$) be the normalized eigenvalues of $F$. We refer the reader to \cite{An} for more details.


The degree $5$ standard $L$-function $D(s,F)  =L(s, \pi_F, \rho_5)$ attached to $F$ corresponds to the first non-trivial (degree $5$) representation $\rho_5 \colon \Sp(4,\C) \to \SO(5,\C)$,  where $\pi_F$ denotes the automorphic representation attached to the eigenform $F$ and $\Sp(4,\C)$ is the dual group of $\Sp(4, \R)$.  In particular, it arises as a general $L$-function in the sense of Langlands. 
It is given by
\begin{equation*}
D(s, F)=\prod_{p} D_{F,p} (p^{-s}),
\end{equation*}
where $p$-th Euler factor $L_{F,p}(s)$ of $( s, F)$ is given by
\begin{equation}\label{standardp}
D_{F,p}(t)^{-1}=(1-t)(1-\alpha_{1,p}t)(1-\alpha_{2,p}t)(1-\alpha_{1,p}^{-1}t)(1-\alpha_{2,p}^{-1}t).
\end{equation}
Here $\alpha_{0,p},\alpha_{1,p},\alpha_{2,p}$ denote the Satake  parameters at $p$ attached to $F$ satisfying
$\alpha_{0,p}^2\alpha_{1,p}\alpha_{2,p}=1.$

Let $F\in \mathcal{B}_k$. We may assume that $a_F(I) \not=0$, otherwise $\omega_F = 0$, and there is no contribution to $\mathcal{M}_k$, defined in Theorem \ref{thm1}.  Then for $\re(s)$ sufficiently large,  the standard $L$-function can also be written as the Dirichlet series (\cite{An}) 
\begin{equation}\label{DF}
D(s, F) = \frac{1}{a_F(I)} L(s+1, \chi_{-4})\zeta(2s) \sum_{M \in {\rm SL}_2(\mathbb{Z}) \backslash M_2^+(\mathbb{Z})} \frac{a_F(MM^{\top})}{(\det M)^{s+1/2}},
\end{equation}
and it has meromorphic continuation to the entire complex plane. This follows for instance by unfolding an integral of $F$ against a product of a theta series and a real analytic Siegel Eisenstein series, see e.g., \cite[Prop.~2.3]{AK}, but we do not use this. Here we normalize the Fourier coefficients to be roughly of size one on average, i.e.\ we write the Fourier expansion of $F$ as 
\[ F(Z)=\sum_{M \in \mathcal{S}} a_F(M) (\det M)^{\frac{k}{2} - \frac{3}{4}} e(\text{tr}(MZ)),
\]
where $\mathcal{S}$ denotes the set of positive-definite half-integral $2\times 2$ matrices over $\mathbb{Z}$.

The local factor of $D(s,F)$ at infinity is given by
\begin{align}\label{Linf}
D_{F,\infty}(s) &= \pi^{-5s/2}\Gamma\Big(\frac{s}{2}\Big) \Gamma\Big(\frac{s+k}{2}\Big)\Gamma\Big(\frac{s+k-1}{2}\Big)^2 \Gamma\Big(\frac{s+k-2}{2}\Big) \\
&= 2^{5 - 2k-2s} \pi^{1 - \frac{5s}{2}} \Gamma\Big(\frac{s}{2}\Big) \Gamma(k+ s - 2)\Gamma(k+ s - 1),
\end{align}
and we have $D(s, F) D_{F,\infty}(s)  = D(1-s, F) D_{F,\infty}(1-s) $. The function $D(., F)$ is holomorphic except if $F$ is a Saito-Kurokawa lift of an elliptic modular form $f$, in which case it has a pole at $s=1$ and equals
\begin{equation}\label{lift}
\zeta(s) L(s+1/2, f) L(s - 1/2, f), 
\end{equation}
as a comparison of Euler products shows.

For $F\in \mathcal{B}_k$ we have an approximate functional equation (\cite[Section 5.2]{IK})
\begin{equation}\label{approx}
a_F(I) D(1/2, F)  = 2\sum_{r, s} \frac{\chi_{-4}(s)}{rs^{3/2}} \sum_{M \in {\rm SL}_2(\mathbb{Z}) \backslash M_2^+(\mathbb{Z})} \frac{a_F(MM^{\top})}{\det M} W\Big( \frac{r^2 s \det M}{k^2}\Big)
\end{equation}
where
$$W(x) = \int_{(2)} (k^2 x)^{-u} \frac{D_{F,\infty}(u+1/2)}{D_{F,\infty}(1/2)}e^{u^2} \frac{1-4u^2}{u} \frac{du}{2\pi i}.$$
Note that $W(0) = 1$ and $W(x) \ll_A (1+x)^{-A} $ for any $A> 0$, uniformly in $k$. The factor $(1 - 4u^2)$ cancels possible poles in the case of Saito-Kurokawa lifts.

\subsection{The Kitaoka formula}

We follow \cite{Bl}, which is based on \cite{Ki}.  We generally write $\ell = k - 3/2$. 

For $Q, T \in \mathcal{S}$ and an invertible matrix $C \in M_2(\mathbb{Z})$ denote by 
\begin{equation}\label{klooster}
K(Q, T; C) = \sum e\left(\text{tr}(A C^{-1} Q + C^{-1} D T)\right)
\end{equation}
the ``Kloosterman sum'', where the sum is taken over matrices $\left(\begin{smallmatrix} A & \ast \\ C & D\end{smallmatrix}\right) \in {\rm Sp}_4( \mathbb{Z})$ for a given value of $C$ in a system $X(C)$ of representatives for $\Gamma_{\infty} \backslash {\rm Sp}_4( \mathbb{Z}) /\Gamma_{\infty}$ with $\Gamma_{\infty} = \left\{ \left(\begin{smallmatrix} I & X \\  & I\end{smallmatrix}\right)  \mid X = X^{\top}\right\}.$  

For a real, diagonalizable matrix $P$ with positive eigenvalues $s_1^2, s_2^2$ ($s_1, s_2 > 0$) we write 
\begin{equation}\label{defJ}
    \mathcal{J}_{\ell}(P) := \int_0^{\pi/2} J_{\ell}(4\pi s_1 \sin \theta)J_{\ell}(4\pi s_2 \sin \theta)\sin\theta\, d\theta. 
    \end{equation}

For two matrices $P = \left(\begin{smallmatrix} p_1& p_2/2\\ p_2/2 & p_4\end{smallmatrix}\right) \in \mathcal{S}$, $S = \left(\begin{smallmatrix} s_1& s_2/2\\ s_2/2 & s_4\end{smallmatrix}\right) \in \mathcal{S}$ and $c \in \mathbb{N}$ we define another Kloosterman-type sum
\begin{align} \label{Hkloos}
H^{\pm}(P, S; c) =  \delta_{s_4 = p_4} \underset{d_1 \, (\text{mod }c)}{\left.\sum\right.^{\ast}} \sum_{d_2 \, (\text{mod }c)} e \left( \frac{\overline{d_1} s_4 d_2^2 \mp \overline{d_1} p_2d_2 + s_2d_2 + \overline{d_1}p_1 + d_1 s_1 }{c} \mp \frac{p_2s_2}{2 cs_4}\right).
\end{align}
 
Let   
 \begin{equation}\label{ck}
      c_k =  \frac{1}{4}   \pi^{1/2}   (4\pi)^{3-2k} \Gamma(k-3/2)\Gamma(k-2).
      \end{equation}
 
 Then  for $T, Q\in \mathcal{S}$ and even $k \geq 6$ we have the absolutely convergent formula
  \begin{align}\label{Kita}
8 c_k & \sum_{F \in \mathcal{B}_k^{(2)}} \frac{a_F(T) a_F(Q)}{\| F \|^2}   = \delta_{Q \sim T} \#\text{{\rm Aut}}(T)\\
& +   \sum_{\pm} \sum_{s, c\geq 1} \sum_{U, V} \frac{(-1)^{k/2} \sqrt{2} \pi}{c^{3/2}s^{1/2}} H^{\pm} (UQU^{\top}, V^{-1} T V^{-\top} , c)J_{\ell}  \Big(\frac{4\pi \sqrt{\det(TQ)}}{c s}\Big)\\
& + 8\pi^2   \sum_{\det C \not= 0} \frac{K(Q, T; C)}{|\det C|^{3/2}}  \mathcal{J}_{\ell}(T C^{-1} Q C^{-\top}),
\end{align}
where the sum over $U, V \in {\rm GL}_2(\mathbb{Z})$ in the second term on the right hand side is over matrices 
$$U  = \left(\begin{matrix} * & *\\ u_3 & u_4 \end{matrix}\right)/\{\pm 1\},  \quad  
V  = \left(\begin{matrix} v_1 & *\\ v_3 &* \end{matrix}\right), \quad  (u_3\, u_4) Q \left(\begin{matrix} u_3\\ u_4\end{matrix}\right) = (-v_3\, v_1) T \left(\begin{matrix} -v_3\\ v_1\end{matrix}\right) = s,
$$ 
$Q \sim T$ means equivalence in the sense of quadratic forms and $\text{{\rm Aut}}(T) = \{U \in {\rm GL}_2(\mathbb{Z}) \mid U^{\top} T U = T\}$.   
 
 \subsection{Exponential sums}
 We now prove a lemma which bounds the Kloosterman-type sums $H^{\pm}(P, S; c)$ defined in \eqref{Hkloos} which appear in the rank $1$ part of the Kitaoka's formula \eqref{Kita} by essentially exhibiting a square-root cancellation.
 \begin{lemma}\label{exp} With the notation as above we have
 \[ |H^{\pm}(P, S; c)| \leq \tau(c) c (\det 2P, \det 2S, c) \]
for odd $c$,  where $\tau(c)$ denotes the number of divisors of $c$. 
 \end{lemma}

\begin{proof} By twisted multiplicativity of exponential sums it is enough to treat the case where $c = p^k$ is an odd prime power. 
Write $\nu = \min(k, v_p(s_4))$. Evaluating a quadratic Gau{\ss} sum, we have 
\begin{align}
&\sum_{d_2 \, (\text{mod }p^k)} e \left( \frac{\overline{d_1} s_4 d_2^2 \mp \overline{d_1} p_2d_2 + s_2d_2  }{p^k} \right) \\
&=  \textbf{1}_{p^{\nu} \mid d_1s_2   \mp p_2 } \sqrt{ p^{k+\nu} \chi_{-4}(p^{k-\nu})}e \left( \frac{ -\overline{4\frac{s_4}{p^{\nu}}  d_1} (\frac{\mp p_2 + d_1 s_2}{p^{\nu}})^2   }{p^{k-\nu}}  \right) \psi \Big(\frac{s_4}{p^{\nu}}d_1\Big)  
\end{align}
where $\psi = \chi_p$ is the unique non-trivial quadratic character modulo $p$ of $k-\nu$ is odd  and $\psi$ is the trivial character if $k-\nu$ is even. 
We compute
$$p^{\nu} \Big( -\overline{4\frac{s_4}{p^{\nu}}  d_1} \Big(\frac{\mp p_2 + d_1 s_2}{p^{\nu}}\Big)^2\Big) +  \overline{d_1}p_1 + d_1 s_1 =  \overline{\frac{4s_4}{p^{\nu}}}\Big( d_1 \frac{\det 2S}{p^{\nu}} + \bar{d}_1 \frac{\det 2P}{p^{\nu}} \mp  \frac{2p_2 s_2}{p^{\nu}}\Big)$$
in $\mathbb{Z}/p^k\mathbb{Z}$ for $p^{\nu} \mid d_1s_2   \mp p_2$, using that $s_4 = p_4$. 

Thus we obtain
\begin{align}
|H^{\pm}(P, S;p^k)|& =  \sqrt{ p^{k+\nu}} \Bigg| \underset{\substack{d_1 \, (\text{mod }p^k) \\p^{\nu} \mid d_1s_2   \pm p_2 } }{\left.\sum\right.^{\ast}} \psi(d_1) e \left( \frac{ d_1\overline{\frac{4s_4}{p^{\nu}}}  \frac{ \det 2S}{p^{\nu}} + \bar{d}_1 \overline{\frac{4s_4}{p^{\nu}}}  \frac{ \det 2P}{p^{\nu}}}{p^k} \right)\Bigg|.  
\end{align}
 Note that $p^{-\nu}   \det 2S$ and $p^{-\nu} \det 2P$ are not necessarily integers, but by construction  the above expression is well-defined for $d_1$ (mod $p^k$). 

 Let $\mu = \min(\nu, v_p(s_2)) = \min(\nu, v_p(p_2))$ and write $s_2' = s_2/p^{\mu}$, $p_2' = p_2/p^{\mu}$.  Let us also write $s_4' = s_4/p^{\nu}$, so that $(s_2'p_2's_4', p) = 1$. 
To avoid ambiguities we extend the $d_1$-sum to a sum modulo $p^{k+\nu}$ and write
\begin{equation}\label{Hsum}
|H^{\pm}(P, S;p^k)| = \sqrt{p^ {k-\nu}} \Big| \underset{\substack{d_1 \, (\text{mod } p^{k+\nu}) \\  d_1  \equiv \pm p_2' \overline{s_2'} \, (\text{mod }p^{\nu - \mu}) }}{\left.\sum\right.^{\ast}}\psi(d_1) e \Big( \frac{ d_1 \overline{4s_4'}    \det 2S  + \bar{d}_1 \overline{4s_4'}  \det 2P }{ p^{k+\nu}} \Big)\Big|.
\end{equation}

If $k = 1$, $\nu = 0$, the $d_1$-sum is a Sali\'e sum and the entire expression is bounded by $2p(p, \det 2S, \det 2P)^{1/2}$.

  If $k+\nu =2\lambda \geq 2$, say, is even (in which case also $k - \nu$ is even  and $\nu - \mu \leq \lambda$), we write $d_1 = a + p^{\lambda} b$ with $a, b$ (mod $p^{\lambda}$) and $a  \equiv \pm p_2' \overline{s_2'} \, (\text{mod }p^{\nu - \mu})$ to see that 
  \begin{align}
|H^{\pm}(P, S;p^k)| &= \sqrt{p^ {k-\nu}} \Big| \underset{\substack{a \, (\text{mod } p^{\lambda}) \\  a \equiv \pm p_2' \overline{s_2'} \, (\text{mod }p^{\nu - \mu}) }}{\left.\sum\right.^{\ast}}  e \Big( \frac{ a  \overline{4s_4'}  \det 2S  + \bar{a}  \overline{4s_4'}  \det 2P }{ p^{2\lambda}} \Big)\\
& \quad\quad\quad\quad\quad  \times \sum_{b\, (\text{mod } p^{\lambda})}e \Big( \frac{ b  \overline{4s_4'} \det 2S  - b\bar{a}^2  \overline{ 4s_4'}  \det 2P }{ p^{ \lambda}} \Big)\Big|\\
 & \leq  \sqrt{p^ {k-\nu}}p^{\lambda} \#\{ a \in (\mathbb{Z}/p^{\lambda}\mathbb{Z})^{\ast} \mid a^2  \overline{4s_4'}  \det 2S  -   \overline{4s_4'}  \det 2P \equiv 0\, (\text{mod } p^{\lambda}) \} \\
 & \leq 2 \sqrt{q^ {k-\nu}} p^{\lambda} (\det 2S, \det 2P, p^{\lambda}) \leq 2p^k (\det 2S, \det 2P, p^k). 
\end{align}

If $k+\nu =2\lambda+1 \geq 3$ is odd (in which case still $\nu -\mu \leq \lambda$), we write again $d_1 = a + p^{\lambda} b$ to see that 
$|H^{\pm}(P, S;p^k)|$ equals 
\begin{align}
& \sqrt{p^ {k-\nu}}  \Big| \underset{\substack{a \, (\text{mod } p^{\lambda}) \\  a \equiv \pm p_2' \overline{s_2'} \, (\text{mod }p^{\nu - \mu}) }}{\left.\sum\right.^{\ast}}\chi_p(a)  e \Big( \frac{ a  \overline{4s_4'} \det 2S  + \bar{a}  \overline{4s_4'}  \det 2P }{ p^{2\lambda+1}} \Big)\\
 &\quad\quad\quad\quad\quad\times\sum_{b\, (\text{mod } p^{\lambda+1})}e \Big( \frac{ b \overline{4s_4'} \det 2S  +(- b\bar{a}^2+ p^{\lambda}b^2 \bar{a}^3)  \overline{4s_4'}  \det 2P  }{ p^{ \lambda+1}} \Big)\Big|\\
 \end{align}
\begin{align}
 & =  \sqrt{p^ {k-\nu}}  p^{\lambda} \Big| \underset{\substack{a \, (\text{mod } p^{\lambda}) \\  a \equiv \pm p_2' \overline{s_2'} \, (\text{mod }p^{\nu - \mu}) \\     a^2  \overline{4s_4'}\det 2S  -   \overline{4s_4'}  \det 2P \equiv 0\, (\text{mod } p^{\lambda})  }}{\left.\sum\right.^{\ast}}\chi_p(a)  e \Big( \frac{ a  \overline{4s_4'} \det 2S  + \bar{a}  \overline{4s_4'}  \det 2P }{ p^{2\lambda+1}} \Big)\\
& \quad\quad\quad\quad\quad\quad \times \sum_{c\, (\text{mod } p)}e \Big( \frac{ b \frac{ \overline{4s_4'}\det 2S  - \bar{a}^2 \overline{4s_4'}  \det 2P}{p^{\lambda}} +  b^2 \bar{a}^3   \overline{4s_4'}  \det 2P  }{ p } \Big)\Big|.
\end{align}
The $c$-sum is of absolute value $p^{1/2}$ unless $  a^2  \overline{4s_4'}\det 2S  -   \overline{4s_4'}  \det 2P $ is not only divisible by $p^{\lambda}$, but by $p^{\lambda+1}$, in which case we bound it trivially by $p$.  We conclude
$$|H^{\pm}(P, S;p^k)| \leq 2\sqrt{p^ {k-\nu}}  p^{\lambda} p^{1/2}  (\det 2S, \det 2P, p^{\lambda+1}) \leq 2 p^k (\det 2S, \det 2P, p^k)$$
since $\lambda+1 \leq k$. 
%
\end{proof}

A similar bound holds for all $c$, but the computation for even prime powers is a little more tedious, and we do not need it. We conclude the following corollary.
\begin{cor}\label{cor1} If $(\det 2P, c) \ll 1$ and $c = q 2^{\kappa}$ with $q$ odd, $\kappa \in \mathbb{N}_0$, then
$$|H^{\pm}(P, S; c)| \ll c^{1+\varepsilon}2^{\kappa}.  $$
 \end{cor}

\subsection{Oscillatory integrals}
At several places in our analysis, we would require good bounds on various types of exponential integrals. We devote this subsection towards this objective.

The following is \cite[Lemma 8.1]{BKY}.
\begin{lemma} \label{integrationbyparts}
 Let $Y \geq 1$, $X, Q, U, R > 0$, 
and suppose that $w$ 
 is a smooth function with support on $[\alpha, \beta]$, satisfying
\begin{equation*}
w^{(j)}(t) \ll_j X U^{-j}.
\end{equation*}
Suppose $h$ 
  is a smooth function on $[\alpha, \beta]$ such that
\begin{equation}
 |h'(t)| \geq R
\end{equation}
for some $R > 0$, and
\begin{equation}\label{diffh0}
h^{(j)}(t) \ll_j Y Q^{-j}, \qquad \text{for } j=2, 3, \dots.
\end{equation}
Then  
\begin{equation*}
 \int_{\mathbb{R}} w(t) e^{i h(t)} dt  \ll_A (\beta - \alpha) X [(QR/\sqrt{Y})^{-A} + (RU)^{-A}].
\end{equation*}
for every $A > 0$. 
\end{lemma}

The following is a classical version of Parseval's identity in many variables, see e.g. \cite[Theorem 1.15]{SW}. For $x=(x_1, \ldots, x_n)$ and $X = (X_1, \ldots, X_n)$ with $X_j \neq 0$ and a function $f \colon \mathbb{R}^n \to \mathbb{C}$, we put $f  (\frac{x}{X}  ):=  f  ( \frac{x_1}{X_1},\ldots, \frac{x_n}{X_n}  )  $.
\begin{lemma} \label{parseval}
Let $A\in {\rm GL}_n(\mathbb{R})$, $b \in \mathbb{R}^n$, $X = (X_1, \ldots, X_n)$ with $X_j \geq 1$. Let $f : \mathbb{R}^n \rightarrow \mathbb{C} $ be a smooth, compactly supported function. We have the following identity.
\begin{align} \label{parseval-eq}
&\int_{\mathbb{R}^n} e(x^{\top} A x + x^{\top} b) f\Big(\frac{x}{X}\Big) dx \\
&= \frac{e(\frac{n}{8} \text{{\rm sgn}} \det A )}{2^{n/2} |\det A|^{1/2}} X_1\cdots X_n \int_{\mathbb{R}^n} e\Big( \frac{1}{4} (b-t)^{\top} A^{-1} (b-t)\Big) \hat{f}(t_1X_1, \ldots t_nX_n) dt.
\end{align}
\end{lemma}

\begin{lemma} \label{parseval-bd}
With the notation as in Lemma~\ref{parseval}, we have
\begin{equation}\label{quadrstatph}
 \int_{\mathbb{R}^n} e(x^{\top} A x + x^{\top} b) f\Big(\frac{x}{X}\Big) dx  \ll  \frac{\| f \|_{\nu, 2}}{|\det A|^{1/2}}
\end{equation}
for $\nu> n/2$, where the implied constant depends only on  $n$ and $\| . \|_{\nu, 2}$ is the usual $W^{\nu, 2}$-Sobolev norm: $\| D^{\alpha} f \|_2 < \infty$  for $|\alpha | \leq \nu$. 

\end{lemma}

\begin{proof}
    This follows immediately from Lemma \ref{parseval} via the Cauchy-Schwarz inequality to the functions $(1+ \|x\|^2)^\nu$ and $(1+ \|x\|^2)^{-\nu} \hat{f}(t_1X_1, \ldots t_nX_n) $.
\end{proof}

\subsection{Special functions}

For $\ell \in \frac{1}{2} \mathbb{Z} \setminus \mathbb{Z}$, $\ell > 0$ and $\Re z>0$ we have \cite[8.411.13]{GR}
\begin{equation}\label{bes}
J_{\ell}(z) = \frac{1}{\pi} \int_0^{\pi} \cos(k\theta - z \sin\theta) d\theta - \frac{1}{\pi} \int_0^{\infty} e^{-\ell \theta - z\sinh\theta} d\theta =: J_{\ell}^{(1)}(z) + J_{\ell}^{(2)}(z),
\end{equation}
say.  We note the following bound on $J_{\ell}^{(2)}(z)$.
\begin{lemma}
With the above notation, we have
\begin{equation}\label{J2} 
   \Big(z \frac{d}{dz}\Big)^j J_{\ell}^{(2)}(z) \ll_j \ell^{-1}
\end{equation}
for all $j \geq 0$. 
\end{lemma}

\begin{proof}
This is clear for $j=0$, and to see it in general, we claim that  we have
\[  \Big(z \frac{d}{dz}\Big)^j J_{\ell}^{(2)}(z) = -\frac{1}{\pi} \int_0^{\infty} e^{-\ell \theta - z\sinh\theta} f_j(  \theta) d\theta \]
where $f_0(\theta) = 1$ and 
\[ f_{j+1}( \theta) = -f_j(\theta)\Big( \frac{1}{\cosh^2\theta} - \ell \tanh\theta\Big) - f'_j(\theta)\tanh\theta \]
satisfying $f^{(\nu)}_j(\theta) \ll_{j, \nu} (1 + \ell \theta)^j \ell^{\nu}.$
This implies \eqref{J2} by induction. The recursive formula follows from partial integration via
\begin{align}
&z \frac{d}{dz} \int_0^{\infty} e^{-\ell \theta - z\sinh\theta} f_j(  \theta) d\theta = -z\int_0^{\infty} e^{-\ell \theta - z\sinh\theta}\cosh(\theta) \tanh(\theta) f_j(  \theta) d\theta \\
& = 
- \int_0^{\infty} e^{  - z\sinh\theta}\frac{d}{d\theta} \big(e^{-\ell \theta} \tanh(\theta) f_j(  \theta)\big) d\theta . \qedhere
\end{align}
\end{proof}

We will often use that for small $x$ the Bessel function is negligible: we have
\begin{equation}\label{bes1}
  J_{\ell}(x) \ll 2^{-\ell}, \quad x \leq \ell/2, 
\end{equation}
which follows easily from \cite[8.411.4]{GR}. 
Finally we need   Olver's uniform asymptotic expansion (\cite[(4.24)]{Ol}; see also \cite[(4.4)]{Bl}): we have
\begin{align}\label{unif}
& J_{\ell}(x) =  \frac{\exp(i\psi(x, \ell) ) F_{\ell}^{+}(x) + \exp(-i\psi(x, \ell) ) F^{-}_{\ell}(x)}{(x^2 - \ell^2)^{1/4}} + O(\ell^{-200}), \\
& \psi(x, \ell) =  \sqrt{x^2 - \ell^2} - \ell \arctan(\sqrt{(x/\ell)^2 - 1})
\end{align}
for $x \geq 2\ell$, where $F_{\ell}^{\pm}$ are smooth non-oscillating functions  satisfying the uniform bounds $x^j (F_{\ell}^{\pm})^{(j)}(x) \ll_j 1$ for all $j \in \mathbb{N}_0$.  

We recall the important formula (\cite[2.12.20]{PBM}\footnote{Note that a factor $\pi$ is missing in \cite[2.12.20.5]{PBM}  in comparison with \cite[2.12.20.2/3]{PBM}.})
\begin{equation}\label{bes4}
2\Re \Bigl[ e\Big(- \frac{\ell+1}{4}\Big) \int_0^{\infty} e\Big((\alpha+\beta)z + \frac{\gamma}{z}\Big)J_{\ell}(4\pi \sqrt{\alpha \beta}z )\frac{dz}{z}\Bigr]    = 2\pi J_{\ell}(4\pi \sqrt{\alpha \gamma}) J_{\ell}(4\pi \sqrt{\beta \gamma}) 
\end{equation}
for $\alpha, \beta, \gamma > 0$. 


For $z > 0$ we consider the function
\begin{equation}\label{defE}
E(z) := \int_0^{\pi/2}  e\Big(\frac{\sin^2\theta}{z}\Big) \sin\theta\, d\theta =  e\Big(\frac{1}{z}\Big)z^{1/2}\int_0^{z^{-1/2}} e(-t^2) dt. 
\end{equation}
For $z > 1$, this is a flat function in the sense that  
\begin{equation}\label{Esmall}
   z^n E^{(n)}(z) \ll_n  1, \quad z > 1
   \end{equation}
   for $n \in \mathbb{N}_0$. 
   For $z \leq 1$, we write
\begin{equation}\label{Ebig}
E(z) = e\Big(\frac{1}{z}\Big)z^{1/2}\Big( \frac{1-i}{4} - \int_{z^{-1/2}}^{\infty}  e(-t^2) dt\Big) =: E_1(z) + E_2(z),
\end{equation}
say. We have 
\[ E_2(z) = - e\Big(\frac{1}{z}\Big) z^{1/2}  \int_{z^{-1/2}}^{\infty}  e(-t^2) dt = -\frac{1}{2} \int_0^{\infty} \frac{e(-t/z)}{\sqrt{t+1}} dt = -\frac{iz }{4\pi}\Big(1 + \int_0^{\infty} \frac{e(-t/z)}{2(t+1)^{3/2}} dt \Big) \]
Coupling each differentiation in $z$ with partial integration in $t$, we see that this term
\begin{equation}\label{Ebig2}
z^nE^{(n)}_2(z)  \ll_n z , \quad z \leq 1
\end{equation}
for all $n \in \mathbb{N}_0. $ In particular we have 
\begin{equation}\label{Egen}
E(z) \ll \min(1, z^{1/2})
\end{equation}
 in all cases.

\section{The diagonal term and the contribution of the lifts} \label{diag-section}

Let $\varepsilon> 0$ be arbitrarily small. To simplify the notation we will henceforth write $A \preccurlyeq B$ to mean $A \ll  k^{\varepsilon} B$, and similarly $A \succcurlyeq B$ for $A \gg_{\varepsilon} Bk^{-\varepsilon}$. We combine both directions in the notation $A \approxeq B$ to mean $B \preccurlyeq A\preccurlyeq B$. The value of $\varepsilon$  may change from line to line in a linear fashion, i.e.\ we may replace $\varepsilon$ by $5\varepsilon$, say, when we use the notations $\preccurlyeq$, $\succcurlyeq$  and $\approxeq$. 


We now start with the proof of Theorem \ref{thm1}. We substitute the approximate functional equation \eqref{approx} into the definition of $\mathcal{M}_k$ getting
\begin{align}
\mathcal{M}_k 
&= 8c_k\sum_{r, s} \frac{\chi_{-4}(s)}{4rs^{3/2}} \sum_{M \in {\rm SL}_2(\mathbb{Z}) \backslash M_2^+(\mathbb{Z})} \frac{1}{\det M} W\Big( \frac{r^2 s \det M}{k^2}\Big)  \sum_{F \in \mathcal{B}_k^{(2)}} \frac{a_F(I)a_F(MM^{\top})}{\| F \|^2}
\end{align}
with $c_k$ as in \eqref{ck}. Next we insert the formula \eqref{Kita} and start with the analysis of diagonal term. Since $\text{Aut}(I) = 8$ it equals
\begin{align}
2 \sum_{r, s} \frac{\chi_{-4}(s)}{rs^{3/2}}   W\Big( \frac{r^2 s }{k^2}\Big)& =2\int_{(2)} \zeta(1 + 2u)L(3/2 + u, \chi_{-4}) \frac{D_{F,\infty}(u+1/2)}{D_{F,\infty}(1/2)}e^{u^2} \frac{1-4u^2}{u} \frac{du}{2\pi i}\\
& = 2 L(3/2, \chi_{-4}) \log k  + c_0 + O(1/k)
\end{align}
with
\[c_0 = L(3/2, \chi_{-4}) \Big(\frac{3}{2} \gamma - \frac{\pi}{4} -\frac{7}{2} \log 2 - \frac{5}{2} \log \pi\Big) + L'(3/2, \chi_{-4}) \]
as one can see by a contour shift to $\Re u = -1/2$, picking up the double pole at $u = 0$ and recalling the definition of $D_{F,\infty}$ in \eqref{Linf}. To derive the value of the  constant $c_0$ one can use $\frac{\Gamma'}{\Gamma}(1/4) = -\gamma - \frac{\pi}{2} - \log 8$ by \cite[8.364.6]{GR} and $\frac{\Gamma'}{\Gamma}(k) = \log k + O(1/k)$ by Stirling's formula. 


For the contribution to $\mathcal{M}_k$ of the lifts $F$ coming from an elliptic form $f$  we recall \eqref{lift} as well as the fact that in this case we have $$\omega_F \asymp \frac{1}{k^3} \frac{L(1/2, f \times \chi_{-4})}{L(1,\text{sym}^2 f)}$$
(cf.\ \cite[Theorem 3.11]{DPSS}). Thus the contribution of the lifts is at most
$$\ll \sum_f  \frac{1}{k^3}  \frac{L(1/2, f \times \chi_{-4}) L(0, f) L(1, f)}{L(1, \text{sym}^2 f)}  \preccurlyeq k^{  - 1/2}.$$
Here we use $L(0, f)L(1, f) \ll k L(1, f)^2 \preccurlyeq k $ by the functional equation, $1/L(1, \text{sym}^2 f) \preccurlyeq 1$, and $L(1/2, f \times \chi_{-4}) \preccurlyeq k^{1/2 }$ by the convexity bound, and the dimension of the space of elliptic cusp forms of weight $k$ is $O(k)$. 


\section{The rank one term}\label{rank1}

The rank one term equals
\begin{align}
\mathcal{R}_1 :=  \sum_{r, s} &\frac{\chi_{-4}(s)}{4rs^{3/2}} \sum_{M \in {\rm SL}_2(\mathbb{Z}) \backslash M_2^+(\mathbb{Z})} \frac{1}{\det M} W\Big( \frac{r^2 s \det M}{k^2}\Big) \sum_{\pm} \sum_{d, c} \sum_{U, V} \frac{(-1)^{k/2}\sqrt{2}\pi}{c^{3/2} d^{1/2}} \\
 &\times H^{\pm} (U MM^{\top} U^{\top}, V^{-1}V^{-\top}, c) J_{\ell} \Big(\frac{4\pi \det M}{cd}\Big)
 \end{align}
with the same summation conditions as in \eqref{Kita}. 

The matrix $UM$ runs through $\Gamma_{\infty} \backslash M^+_2(\mathbb{Z})$. Let us fix $\Delta = \det M$. Then $UM = (\begin{smallmatrix}\alpha  & \beta\\   \gamma&  \delta\end{smallmatrix})$, where $(\gamma, \delta) \mid \Delta$. If $\gamma \not= 0$, then $\alpha$ runs modulo $\gamma$ and satisfies $\alpha \equiv \frac{\Delta}{(\gamma, \delta)} \overline{\frac{\delta}{(\gamma, \delta)}}$ (mod $\frac{\gamma}{(\gamma, \delta)}$), and then $\beta$ is determined by $\Delta = \alpha \delta - \beta \gamma$. If $\gamma = 0$, then $\alpha = \Delta/\delta$ and $\beta$ runs modulo $\delta$. In either case, $(\alpha, \beta)$ runs through a set $S_{\Delta}(\gamma, \delta)$ of cardinality $(\gamma, \delta)$. Let $T(\gamma, \delta)$ be the set of matrices $V =  \left(\begin{smallmatrix} v_1 & *\\ v_3 &* \end{smallmatrix}\right)$ as above such that $v_1^2 + v_3^2 = \gamma^2 + \delta^2$, so its cardinality is $\ll (\gamma^2 + \delta^2)^{\varepsilon}$. 

With this notation we obtain the bound
\begin{align}
\mathcal{R}_1\ll \Big| \sum_{r, s} &\frac{\chi_{-4}(s)}{rs^{3/2}} \sum_{\Delta}  \sum_{(\gamma, \delta) \mid \Delta} \sum_{(\alpha, \beta) \in S_{\Delta}(\gamma, \delta)}
 \sum_{V \in T(\gamma, \delta)} \frac{1}{\Delta} W\Big( \frac{r^2 s \Delta}{k^2}\Big)   \sum_{c}   \frac{1}{c^{3/2} (\gamma^2 + \delta^2)^{1/2}} \\
 &\times H^{\pm} \Big(\left(\begin{smallmatrix}\alpha^2 + \beta^2& \alpha\gamma + \beta \delta\\ \alpha \gamma + \beta\delta& \gamma^2 + \delta^2 \end{smallmatrix}\right) , V^{-1}V^{-\top}, c\Big) J_{\ell} \Big(\frac{4\pi \Delta }{c(\gamma^2 + \delta^2)}\Big)\Big|. 
 \end{align}
Since $\det V = 1$, we can apply Corollary \ref{cor1}, and we will use the notation $c = q2^{\kappa}$ with $q$ odd, $\kappa \in \mathbb{N}_0$.   With  $\rho = (\gamma, \delta)$, the previous display becomes
\begin{align}
\Big| \sum_{r, s} &\frac{\chi_{-4}(s)}{rs^{3/2}} \sum_{\Delta}\sum_{\rho}  \sum_{(\gamma, \delta)=1} \sum_{(\alpha, \beta) \in S_{\Delta}(\gamma, \delta)}
 \sum_{V \in T(\gamma, \delta)} \frac{1}{\rho\Delta} W\Big( \frac{r^2 s \rho\Delta}{k^2}\Big)   \sum_{c}   \frac{1}{c^{3/2} \rho(\gamma^2 + \delta^2)^{1/2}} \\
 &\times H^{\pm} \Big(\left(\begin{smallmatrix}\alpha^2 + \beta^2& \alpha\gamma + \beta \delta\\ \alpha \gamma + \beta\delta& \gamma^2 + \delta^2 \end{smallmatrix}\right) , V^{-1}V^{-\top}, c\Big) J_{\ell} \Big(\frac{4\pi \Delta }{c\rho(\gamma^2 + \delta^2)}\Big)\Big|. 
 \end{align}



\begin{remark}
To get a feeling for the order of magnitude, consider the case $r = s =\rho = 1$, $c$ odd. Then we are summing
$$\sum_{\Delta \preccurlyeq	 k^2} \frac{1}{\Delta} \sum_c \frac{H(*, *, c)}{c^{3/2}} \sum_{m} \frac{r(m)}{m^{1/2}}   J_{\ell}\Big(\frac{4\pi \Delta}{cm}\Big)$$
where $m = \gamma^2 + \delta^2$ and $r(m)$ is the number of representations of $m$ as a sum of two squares. If we use the approximate bounds $H(c) \ll c$ and  $J_{\ell}(x) \ll x^{-1/2}$, 
 $cm \ll \Delta/k$ from \eqref{bes1} and \eqref{unif},  the trivial bound is  $O(1)$, which just fails. However, if we decompose $J_{\ell} = J_{\ell}^{(1)} + J_{\ell}^{(2)}$ as in \eqref{bes}, then for the portion involving $J_{\ell}^{(2)}$  the trivial bound suffices. For the other part, the saving will come from Poisson summation in $\Delta$ modulo $c$. We now make this strategy precise. 
\end{remark}
First we drop the terms with $\Delta/c\rho(\gamma^2 + \delta^2) \ll k/100$ at the cost of a negligible error by \eqref{bes1}. We remember this by inserting a smooth function $v(\frac{\Delta}{kc\rho(\gamma^2 + \delta^2)})$ where $v(x) = 1$ for $x \geq 1/80$ and $v(x) = 0$ for $x < 1/100$. Then we decompose the Bessel function as in \eqref{bes}. As mentioned above, for the second piece, the trivial bound suffices by Corollary \ref{cor1} and \eqref{J2} to obtain
\begin{equation}\label{1a}
\begin{split}
  \preccurlyeq  & \sum_{r, s} \frac{1}{rs^{3/2}} \underset{c\rho(\gamma^2 + \delta^2) \ll \Delta/k}{\sum_{\Delta \ll \frac{k^2}{r^2s}}\sum_{c= q2^{\kappa}}  \sum_{\rho}  \sum_{\gamma, \delta} } \frac{2^{\kappa}}{\rho \Delta c^{1/2-\varepsilon}  (\gamma^2 + \delta^2)^{1/2 }  } \frac{1}{k}\\
& \preccurlyeq   \sum_{r, s} \frac{1}{rs^{3/2}} \sum_{\Delta \preccurlyeq \frac{k^2}{r^2s}} \frac{1}{\Delta} \Big(\frac{\Delta}{k}\Big)^{1/2} \frac{1}{k} \preccurlyeq k^{-1/2}.
\end{split}
\end{equation}

We now turn towards the contribution involving $J_{\ell}^{(1)}$. 
We split $\Delta$ in residue classes modulo $c$ getting a bound
\begin{align}
\ll \Big| \sum_{r, s} &\frac{\chi_{-4}(s)}{rs^{3/2}} \sum_{c} \sum_{\rho}  \sum_{\substack{\Delta_0\, (\text{mod } c)\\ (c, \rho) \mid \Delta_0 }}  \sum_{(\gamma, \delta) = 1} \sum_{(\alpha, \beta) \in S_{\Delta_0}(\rho\gamma, \rho\delta)}
 \sum_{V \in T(\rho\gamma, \rho\delta)}    \frac{H^{\pm}(*, *, c)}{c^{3/2} \rho(\gamma^2 + \delta^2)^{1/2}} \\
 &
\times  \sum_{\rho \Delta \equiv \Delta_0 \, (\text{mod }c) } \frac{1}{\rho \Delta}v\Big(\frac{\Delta}{kc\rho(\gamma^2 + \delta^2)}\Big) W\Big( \frac{r^2 s \rho \Delta}{k^2}\Big) J_{\ell}^{(1)} \Big(\frac{4\pi \Delta }{c\rho(\gamma^2 + \delta^2)}\Big)\Big|. 
\end{align}
By Poisson summation, the inner sum equals
$$\frac{(\rho, c)}{\rho c} \sum_h e\Big( \frac{h \frac{\Delta_0}{(c, \rho)} \overline{\frac{\rho}{(c, \rho)}}}{c/(c, \rho)}\Big) \int \frac{1}{x}v\Big(\frac{x}{kc\rho(\gamma^2 + \delta^2)}\Big) W\Big( \frac{r^2 s \rho x}{k^2}\Big) J_{\ell}^{(1)} \Big(\frac{4\pi x }{c\rho(\gamma^2 + \delta^2)}\Big) e\Big( - \frac{xh}{c/(c, \rho)} \Big) dx$$
Inserting \eqref{bes}, the integral is a linear combination of two terms (according to $\pm$)
\[ \int_0^{\pi} \int \frac{W(x)}{x}v\Big(\frac{xk}{r^2sc\rho^2(\gamma^2 + \delta^2)}\Big)  e \Big(\pm \ell \theta \mp \frac{2 x  k^2 }{r^2 sc\rho^2(\gamma^2 + \delta^2)} \sin\theta \Big) e\Big( - \frac{xh k^2(c, \rho)}{r^2 s \rho c} \Big) dx \, d\theta. \]
We split the $x$-integral into dyadic ranges $x \asymp X \preccurlyeq 1$ (as $X \succcurlyeq 1$ contributes negligibly by the rapid decay of $W$) and put 
\begin{align} \label{}
\Xi := \frac{Xk^2}{r^2 s c \rho^2 (\gamma ^2 + \delta^2)} \gg k,
\end{align}
 where the lower bound follows from the support of $v$. Integration by parts in the $x$-integral shows first $h\preccurlyeq 1$ 
 (and in fact also $\gamma, \delta \preccurlyeq 1$, but we do not need this information).  


Let us fix a small constant $0 < \eta < 1/2$ and  restrict first to $|\theta - \frac{\pi}{2}| > k^{-\eta}$. Once $h$ is fixed,  integration by parts in the $x$-integral restricts $\sin \theta$ to an interval of length $$\preccurlyeq L := \frac{r^2 s c \rho^2 (\gamma^2 + \delta^2)}{k^2} = \frac{X}{\Xi} \preccurlyeq \frac{1}{k}.$$  By the mean value theorem, this restricts $\theta$  to an interval of length $\preccurlyeq L k^{\eta}$, and   hence bounds the $x, \theta$-integral by $\preccurlyeq L k^{\eta}$.  
We obtain the bound
\begin{align}
\preccurlyeq   \sum_{r, s} &\frac{1}{rs^{3/2}} \underset{r^2 \rho^2 s c (\gamma^2 + \delta^2) \ll k }{\sum_{c = q2^{\kappa}} \sum_{\rho}\sum_{(\gamma, \delta) = 1}  }    \frac{2^{\kappa}}{c^{3/2} \rho(\gamma^2 + \delta^2)^{1/2}} \frac{(\rho, c)}{\rho c} \frac{r^2 s c \rho^2 (\gamma^2 + \delta^2)}{k^2} k^{\eta}  \\
& \times \sum_{h \preccurlyeq 1} \Big|\sum_{V \in T(\rho\gamma, \rho\delta)}   \sum_{\substack{\Delta_0\, (\text{mod } c)\\ (c, \rho) \mid \Delta_0 }}   \sum_{(\alpha, \beta) \in S_{\Delta_0}(\rho\gamma, \rho\delta)} 
H^{\pm}(*, *, c)
e\Big( \frac{h \frac{\Delta_0}{(c, \rho)} \overline{\frac{\rho}{(c, \rho)}}}{c/(c, \rho)}\Big) \Big|  
\end{align}
By Corollary \ref{cor1}, the second line is bounded by $\preccurlyeq c^2 2^{\kappa} \rho$. This gives the total bound
\begin{equation}\label{1b}
\begin{split}
& \preccurlyeq  k^{\eta} \sum_{r, s} \frac{1}{rs^{3/2}} \underset{r^2 \rho^2 s c (\gamma^2 + \delta^2) \ll k }{\sum_{c=q2^{\kappa}} \sum_{\rho}\sum_{(\gamma, \delta) = 1}  }    \frac{2^{\kappa}}{c^{3/2} \rho(\gamma^2 + \delta^2)^{1/2}} \frac{(\rho, c)}{\rho c} \frac{r^2 s c \rho^2 (\gamma^2 + \delta^2)}{k^2}  c^2 \rho \\
&\preccurlyeq  k^{\eta}  \sum_{r, s} \frac{r}{s^{1/2}} \underset{r^2 \rho^2 s c (\gamma^2 + \delta^2) \ll k }{\sum_{c=q2^{\kappa}} \sum_{\rho}\sum_{(\gamma, \delta) = 1}  }        \frac{  2^{\kappa}  (\gamma^2 + \delta^2)^{1/2} c^{1/2} \rho(\rho, c)}{k^2}  \preccurlyeq k^{\eta -1/2}. 
\end{split}
\end{equation}
Let us now consider the remaining portion $|\theta - \frac{\pi}{2}| \ll k^{-\eta}$. If $\Xi \ll k^{1+\eta}$ with a sufficiently small implied constant, then the function $\theta \mapsto \pm \ell \theta \mp \frac{2 x  k^2 }{r^2 sc\rho^2(\gamma^2 + \delta^2)} \sin\theta$ has a derivative of size $\asymp k$. Hence 
Lemma \ref{integrationbyparts}  with 
$$R = k, \quad Y = \Xi, \quad Q = 1, \quad U = k^{-\eta}$$ 
shows that the $\theta$-integral is negligible. Therefore we may assume that $\Xi \gg k^{1+\eta}$. Again, integration by parts restricts $\sin\theta$ to an interval of length $\preccurlyeq L$ as above and hence $\theta$ to an interval of length $ \preccurlyeq L^{1/2}$. To see this, one can write $\theta- \pi/2=\beta$, so that $\sin \theta = \cos \beta \asymp 1- \beta^2/2$, since $\beta$ is very small. In the same way as before we obtain the total bound
\begin{equation}\label{1c}
\begin{split}
& \preccurlyeq   \sum_{r, s} \frac{1}{rs^{3/2}} \underset{r^2 \rho^2 s c (\gamma^2 + \delta^2) \preccurlyeq k^{1-\eta} }{\sum_{c=q2^{\kappa}} \sum_{\rho}\sum_{(\gamma, \delta) = 1}  }    \frac{2^{\kappa}}{c^{3/2} \rho(\gamma^2 + \delta^2)^{1/2}} \frac{(\rho, c)}{\rho c}\Big( \frac{r^2 s c \rho^2 (\gamma^2 + \delta^2)}{k^2}\Big)^{1/2}  c^2 \rho \\
&\preccurlyeq     \sum_{r, s}   \underset{r^2 \rho^2 s c (\gamma^2 + \delta^2) \preccurlyeq k^{1-\eta} }{\sum_{c=q2^{\kappa}} \sum_{\rho}\sum_{(\gamma, \delta) = 1}  }        \frac{ 2^{\kappa} (\rho, c)}{sk}  \preccurlyeq k^{-\eta}. 
\end{split}
\end{equation}
Combining \eqref{1a}, \eqref{1b} and \eqref{1c} and choosing $\eta = 1/4$, we have shown   $\mathcal{R}_1 \preccurlyeq k^{-1/4}$. 
 
\section{The rank two term}\label{rank2} 

With the notation as in \eqref{Kita}, recall that  the rank two term is
\begin{align}
\mathcal{R}_2 :=   \sum_{r, s} \frac{\chi_{-4}(s)}{4rs^{3/2}} \sum_{M \in {\rm SL}_2(\mathbb{Z}) \backslash M_2^+(\mathbb{Z})} \frac{1}{\det M} & W\Big( \frac{r^2 s \det M}{k^2}\Big) \times \\
& \times \sum_{\det C \not= 0} \frac{K(I, MM^{\top}, C)}{|\det C|^{3/2}} \mathcal{J}_k(MM^{\top} C^{-1} C^{-\top}).
\end{align}

\subsection{Initial reductions} 
We start by confirming that this expression is well-defined: if we replace $M$ with $UM$ for some $U \in {\rm SL}_2(\mathbb{Z})$, then $K(I, UMM^{\top}U^{\top}, C) = K(I, MM^{\top}, C U^{-\top})$ (see \cite[p.\ 153]{Ki}) and $$\mathcal{J}_k(UMM^{\top} U^{\top}C^{-1} C^{-\top}) = \mathcal{J}_k( MM^{\top} U^{\top}C^{-1} C^{-\top}U) = \mathcal{J}_k( MM^{\top} (CU^{-\top})^{-1} (CU^{-\top})^{-\top} ), $$
so each term remains unchanged under the move $(M, C) \mapsto (UM, CU^{\top})$.  


For a given value of $d \in \mathbb{Z}\setminus \{0\}$, let $V_d$ be a set of representatives $\{C \in M_2(\mathbb{Z}) \mid \det C = d\}/{\rm SL}_2(\mathbb{Z})$. The cardinality of $V_d$ is polynomial in $d$, and we may assume that the norm of each $J \in V_d$ is polynomial in $d$. If $\det C = d$, we write $C = JC_0$ with $C_0\in {\rm SL}_2(\mathbb{Z})$ and $J \in V_d$.  Let $N = M^{\top}C_0^{-1}$, which runs through $M^+_2(\mathbb{Z})$. Then
$$K(I, MM^{\top}, C)\mathcal{J}_k(MM^{\top} C^{-1} C^{-\top}) = K(I, N^{\top}N, J)\mathcal{J}_k( N J^{-1}J^{-\top} N^{\top}). $$
Note that the Kloosterman sum depends only on $N$ modulo $d$. 

Next we turn to size considerations, using the following two observations: 

(i) From the definition \eqref{defJ} of $\mathcal{J}_k$ and \eqref{bes1},   the smallest eigenvalue of $NJ^{-1}J^{-\top}N^{\top}$, and hence in particular $ \det NJ^{-1}$, must be $\geq k^2/100$, otherwise $\mathcal{J}_k$ is negligible.  

(ii) On the other hand, by the rapid decay of the weight function $W$ we can truncate $r^2 s  \det M = r^2 s \det N  \preccurlyeq k^2$ at the cost of a negligible error, and so we have $\det NJ^{-1} \preccurlyeq k^2/(r^2 s\det J)$. 

Combining these two facts, we can truncate
\[ r, \, s, \, \det J  \preccurlyeq 1, \quad \det N \gg k^2\]
As mentioned above, we can choose a convenient system of representatives for $J$, e.g.\ 
lower triangular matrices, 
so that $\det J \preccurlyeq 1$ also implies $\| J \| \preccurlyeq 1$. 
Moreover, $\det N \gg k^2$ clearly implies $\| N \| \gg k$. Finally note that 
\[  N^{\text{adj}} (N^\top)^{\text{adj}} \le \frac{100 (\det N)^2}{k^2} J^{-1}J^{-\top}\]
(as an inequality of positive definite matrices). 
Taking the trace on both sides, we obtain   $\|N \| = \|N^{\mathrm{adj}}\| \preccurlyeq k$ using (ii).  


\medskip

Hence in order to bound the rank two term, we fix some $J \in \bigcup_{0 \not= d \preccurlyeq 1} V_d$ with $\det J = d$ and then some $N_0 \in M_2(\mathbb{Z}/d  \mathbb{Z})$. It then suffices to estimate 
$$ \frac{1}{k^2}  \sum_{\substack{N \in   M_2^+(\mathbb{Z})\\ N \equiv N_0\, (\text{mod } d)}}  W_1\Big(\frac{\| N \|}{K_1}\Big) W_2\Big(\frac{\det N}{K_2^2}\Big)   \mathcal{J}_k( N J^{-1}J^{-\top} N^{\top}). 
$$
for two fixed smooth functions $W_1, W_2$ with compact support in $(0, \infty)$ and $k \ll K_2 \ll K_1 \preccurlyeq k$. From the weight function it is automatic that $\det N > 0$. 

Next we insert \eqref{bes4} and \eqref{defE} into the definition \eqref{defJ} of $\mathcal{J}_k$ to conclude that it suffices to estimate
\begin{equation}\label{prelim}
\frac{1}{k^2}  \sum_{\substack{N \in   M_2(\mathbb{Z})\\ N \equiv N_0\, (\text{mod } d)}}  W_1\Big(\frac{\| N \|}{K_1}\Big) W_2\Big(\frac{\det N}{K_2^2}\Big)    \int_0^{\infty} e\big(\| N J^{-1}\|^2 z \big) J_{\ell}\big(4\pi z  \det NJ^{-1}\big) E(z) \frac{dz}{z}. 
\end{equation}
      We split the absolutely convergent $z$-integral into dyadic ranges and insert a weight function $W_3(z/Z)$ where $W_3$ has compact support in $[1, 2]$. By \eqref{bes1} we may assume $1/k \preccurlyeq Z$  at the cost of a negligible error. We also note that the function $$W : N = \left(\begin{matrix} n_1 & n_2 \\ n_3 & n_4\end{matrix}\right)\mapsto W_1\Big(\frac{\| N \|}{K_1}\Big) W_2\Big(\frac{\det N}{K_2^2}\Big) $$
  is ``almost flat'' in the sense that
\begin{equation}\label{flat}
 \| n_1^{i_1} \cdots n_4^{i_4} W^{(i_1, \ldots, i_4)} \|_{\infty} \preccurlyeq_{i_1, \ldots, i_4} 1
 \end{equation}
  for all $i_1, \ldots, i_4 \in \mathbb{N}_0$. We will use this frequently in combination with \eqref{quadrstatph} in stationary phase arguments. 
 
 Before we formally analyze the expression \eqref{prelim}, it might be instructive to get a feeling how much saving is required.     The $N$-sum has roughly $k^4$ terms. The generic range is $Z \asymp 1/k$, in which case the Bessel function and the $E$-function are both roughly of size $1/k^{1/2}$, so the trivial bound is (generically) $k$, while the goal is a negative power of $k$. The saving will come from Poisson summation in the $N$-sum, since the integrand is typically highly oscillatory. However, if $NJ^{-1}$ is close to being orthogonal, then $\| N J^{-1}\|^2 \approx \pm 2 \det(NJ^{-1})$, which cancels large parts of the oscillation. In addition, when $z$ is large, the length of the dual sum increases. Both issues have to be treated with  additional care. 
       
 
For the  subsequent analysis we need several  layers of $\varepsilon$-exponents that are still arbitrarily small, but some are larger than others.  To formalize this, for $n = 2, 3, \ldots$, we 
 introduce the notation 
 \begin{equation}\label{layered}
 A \preccurlyeq_{n} B \quad \text{and equivalently} \quad B \succcurlyeq _{n} A
 \end{equation}
to mean $A \ll  B k^{\varepsilon^{1/n}}$. 
 Again the value of $\varepsilon$   may change from line to line in a linear fashion, but $\varepsilon$ is assumed to be so small that 
all implicit and explicit instances $\varepsilon^{1/n}$    are always assumed to be bigger than the explicit and implicit instances of $\varepsilon^{1/m}$ for $m < n$.

\subsection{The generic range}
      
We start the analysis with the generic range where $$1/k \preccurlyeq Z \preccurlyeq_{3} 1/k.$$ 

We insert \eqref{bes} and start with the contribution of $J_{\ell}^{(2)}$ where we recall  \eqref{J2}. We apply Poisson summation in $N$, so that it remains to bound
\begin{align}
&\frac{1}{k^2} \sum_{H \in M_2(\mathbb{Z})}  \Big|   \int_0^{\infty} W_3\Big(\frac{z}{Z}\Big)E(z) \\
   & \times \int_{M_2(\mathbb{R})} W_1\Big(\frac{\| N \|}{K_1}\Big) W_2\Big(\frac{\det N}{K_2^2}\Big) J^{(2)}_{\ell}\big(4\pi z  \det NJ^{-1} \big) e\Big(\| N J^{-1}\|^2 z +\frac{1}{d}\text{tr} N H^{\top}\Big) dN \,  \frac{dz}{z}\Big| . 
\end{align}
The measure $dN$ is the usual entrywise Lebesgue measure. Integration by parts shows quickly that $\| H \| \preccurlyeq_{3} 1$ in the present ranges $$\| N \|\approxeq k^{2}, \quad 1/k \preccurlyeq z \preccurlyeq_{3} k^{-1 }, \quad \| J \| \approxeq 1.$$ Now we can apply a   stationary phase argument in form of the bound \eqref{quadrstatph} with $A = J^{-1} J^{-\top} z$ and 
$$f :  N = \left(\begin{matrix} n_1 & n_2 \\ n_3 & n_4\end{matrix}\right)\mapsto W_1\Big(\frac{\| N \|}{K_1}\Big) W_2\Big(\frac{\det N}{K_2^2}\Big) J^{(2)}_{\ell}\big(4\pi z  \det NJ^{-1} \big)$$
which by \eqref{J2} and \eqref{flat} satisfies
$$ \| n_1^{i_1} \cdots n_4^{i_4} f^{(i_1, \ldots, i_4)} \|_{\infty} \preccurlyeq_{3} k^{-1}.$$
  for all fixed $i_1, \ldots, i_4 \in \mathbb{N}_0$. Recalling in addition \eqref{Egen}, we obtain   the bound
\begin{equation}\label{2a}
\preccurlyeq \frac{1}{k^2}  \int_{z \asymp Z} z^{1/2}  \frac{1}{k} z^{-2} \, \frac{dz}{z} \preccurlyeq_{3} k^{-3/2}.
\end{equation}

Next we consider the contribution of $J_{\ell}^{(1)}$. By the same argument involving Poisson summation we have to bound
\begin{align}
&\frac{1}{k^2} \sum_{H \in M_2(\mathbb{Z})}  \Big|  \int_0^{\pi} \int_0^{\infty} W_3\Big(\frac{z}{Z}\Big)E(z)  
    \int_{M_2(\mathbb{R})} W_1\Big(\frac{\| N \|}{K_1}\Big) W_2\Big(\frac{\det N}{K_2^2}\Big)  \\
    & \times e\Big( \pm\frac{1}{2\pi}  \ell \theta\mp 2  (\det N J^{-1}) z \sin\theta  +\| N J^{-1}\|^2 z +\frac{1}{d}\text{tr} N H^{\top}\Big) dN \,  \frac{dz}{z}\, d\theta . 
\end{align}
Again we see by partial integration that we can truncate $\| H \| \preccurlyeq_{3} 1$ at the cost of a negligible error. In the $\theta$-integral
$$\int_0^{\pi} e\Big( \pm\frac{1}{2\pi}  \ell \theta\mp 2 ( \det N J^{-1})z \sin\theta \Big) d\theta $$
 we first consider  the portion $|\theta - \pi/2| \leq P = k^{-\varepsilon^{1/4}}$. Then  $z (\det NJ^{-1}) \cos \theta \preccurlyeq_{3} kP$ is strictly dominated by $k$. 
We extract this 
region in a smooth way, i.e.\ by inserting a smooth weight function $v( |\theta - \pi/2| /P)$ where $v$ has support in $[-2, 2]$. By Lemma  \ref{integrationbyparts} with 
$$R = k, \quad Y = k, \quad Q = 1, \quad U = P$$
we see that this portion  is negligible. Next we consider the remaining part with a weight function  $1-v( |\theta - \pi/2|/P)$ included. Here we apply a stationary phase argument in the $N$-integral 
\begin{align}
&   \int_{M_2(\mathbb{R})} W_1\Big(\frac{\| N \|}{K_1}\Big) W_2\Big(\frac{\det N}{K_2^2}\Big)  e\Big( \mp 2 (\det NJ^{-1}) z \sin\theta  +\| N J^{-1}\|^2 z + \frac{1}{d}\text{tr} N H^{\top}\Big) dN\\
     & = d\int_{M_2(\mathbb{R})} W_1\Big(\frac{\| N J \|}{K_1}\Big) W_2\Big(\frac{d^2 \det N}{K_2^2}\Big)  e\Big( \mp 2z \det N   \sin\theta  +\| N  \|^2 z +\frac{1}{d}\text{tr} N (HJ^{\top})^{\top}\Big) dN
\end{align}
using \eqref{quadrstatph} with the matrix
\[A = z \left( \begin{matrix} 1 & & & \mp\sin\theta \\ & 1& \pm \sin\theta & \\ & \pm \sin\theta & 1 \\ \mp \sin\theta & & & 1\end{matrix}\right)\]
with  $ \det A  =  z^4 \cos^4\theta $ and recalling \eqref{flat}.  
Using again \eqref{Egen} in connection with \eqref{quadrstatph}, we obtain the bound
\begin{equation}\label{2b}
\preccurlyeq  \frac{1}{k^2} \int_{ |\theta - \pi/2| \geq P} \int_{z \asymp Z} z^{1/2} \frac{1}{ z^2 \cos^2\theta}\frac{dz}{z} \preccurlyeq_{4} k^{-1/2}.
\end{equation}
The bounds \eqref{2a} and \eqref{2b} are clearly acceptable. 

\subsection{The other ranges for \texorpdfstring{$Z$}{Z}}

Let us now assume that
\begin{equation}\label{Z}
   Z \gg  k^{-1+\varepsilon^{1/3}}. 
   \end{equation} 
 We will show in this subsection that   the corresponding parts of \eqref{prelim} are negligible, but this requires some work. 
   
\subsubsection{Preliminary reductions} The condition \eqref{Z} implies first that we can insert the uniform asymptotic expansion \eqref{unif} in \eqref{prelim}, the error being negligible.  We write 
$$W^{\pm}_4(z; N) := W_3(z) \frac{F^{\pm}_{\ell}(4\pi zZ\det NJ^{-1})}{((4\pi zZ \det NJ^{-1})^2 - \ell^2)^{1/4}} Z^{1/2} k$$ with the notation as in \eqref{unif} and observe that $W_4$ is again ``almost flat'' in all variables: 
\begin{equation}\label{flatnew}
  \| z^{i_0} n_1^{i_1} \cdots n_4^{i_4} W_4^{(i_0; i_1, \ldots, i_4)} \|_{\infty} \preccurlyeq_{3} 1
  \end{equation}
for all fixed  $i_0, \ldots, i_4 \in \mathbb{N}_0$. With this notation, we start by investigating the $z$-integral, which is a $\pm$-linear combination of  
\begin{align}\label{zint}
&\quad\quad \frac{1}{Z^{1/2} k} \int_0^{\infty} W_4\Big(\frac{z}{Z}; N\Big) E(z) \\
 &\times e\Big[\| N J^{-1} \|^2z \pm 2 \Big(\sqrt{(z \det NJ^{-1})^2 - \ell_0^2}  -  \ell_0 \arctan\Big(\frac{1}{\ell_0}\sqrt{ (z\det NJ^{-1} )^2 - \ell_0^2}\Big) \Big) \Big]\frac{dz}{z}
 \end{align}
where $\ell_0 = \ell/(4\pi)$. We recall that for $Z \ll 1$ the function $E$ consists of two terms, one of which has a phase $e(1/z)$ (see \eqref{Ebig}), while for $Z \gg 1$, the function $E$ is flat (see \eqref{Esmall}).  

With this in mind, let us define
\[\phi_{\pm}(z) = \frac{\sigma}{z} + \| N J^{-1} \|^2z \pm 2 \Big(\sqrt{(z \det NJ^{-1})^2 - \ell_0^2}  -  \ell_0 \arctan\Big(\frac{1}{\ell_0}\sqrt{ (z\det NJ^{-1} )^2 - \ell_0^2}\Big) \Big)\]
for $\sigma \in \{0, 1\}$. We compute
\[ \phi_{\pm}'(z) = -\frac{\sigma}{z^2} + \| N J^{-1} \|^2 \pm  2 \frac{\sqrt{(z \det NJ^{-1})^2 - \ell_0^2} }{z} \]
and
\[  \phi^{(j)}_{\pm}  \ll \frac{1}{Z^{j+1}} + \frac{k^2}{Z^{j+1} (\det NJ^{-1})}\preccurlyeq \frac{1}{Z^{j+1}} \]
for all fixed  $j \geq 2$.  In the plus-case, we have $    \phi'_{+}  \succcurlyeq k^2$ 
and hence by Lemma \ref{integrationbyparts} with  (recall \eqref{flatnew})$$R \succcurlyeq k^2,  \quad U\succcurlyeq_{3} Z, \quad  Q \succcurlyeq Z, \quad Y \preccurlyeq \frac{1}{Z}$$
the $z$-integral is negligible by \eqref{Z}. 

In the minus case, by a Taylor argument we have $$\phi_{-}'(z) = - \frac{\sigma }{z^{2}} + \| N J^{-1} \|^2 -  2 |\det NJ^{-1} |+ O\Big(\frac{k^2}{Z^2 \det NJ^{-1}}\Big).$$ Suppose that  $\| NJ^{-1} \|^2 - 2 |\det NJ^{-1} | \geq k^{\varepsilon^{1/n}} (1+Z^{-2})$ where $n = 2$ if $Z \leq k^{-1/10}$ and $n = 4$ if $Z > k^{-1/10}$.  Then by  Lemma  \ref{integrationbyparts} with $$R\geq k^{\rho} (1+Z^{-2}),  \quad U   \succcurlyeq_{3} Z, \quad Q \succcurlyeq Z, \quad Y \preccurlyeq \frac{1}{Z}$$
we see that the $z$-integral is negligible. So from now on we restrict to the minus case, and we assume 
\begin{equation}\label{24}
\| NJ^{-1} \|^2 - 2 |\det NJ^{-1} | \preccurlyeq_{n} 1+ Z^{-2}
\end{equation}
with  $n = 2$ if $Z \leq k^{-1/10}$ and $n = 4$ if $Z > k^{-1/10}$. 
 Note that depending on the sign of $\det N$, the condition in \eqref{24} becomes a union of two conditions. The next section should be read with this in mind. 

\subsubsection{Coordinates} We now introduce suitable coordinates to measure how far $NJ^{-1}$ is from being orthogonal. In the vector space $V := M_2(\mathbb{Q})$ we introduce the two subspaces
$$W_{\pm} := \{N \in V : \| NJ^{-1} \| \pm 2\det NJ^{-1} = 0\}.$$     
One checks directly that 
\[ W_{\pm} = \left\{ \left(\begin{matrix} n_1 & n_2\\ n_3 & n_4 \end{matrix}\right) \in V :   \left(\begin{matrix} n_3\\ n_4 \end{matrix} \right)  = \mp \frac{1}{\det J} J^{\top}J \left(\begin{matrix} -n_2\\ n_1 \end{matrix} \right) =:\left(\begin{matrix} j_{1, \pm}(n_1, n_2) \\ j_{2, \pm}(n_1, n_2) \end{matrix} \right) \right\},\]
and it is clear that $V = W_+ \oplus W_-$.  To see that $W_{\pm}$ are indeed vector spaces, denote $A:=NJ^{-1}$. Then the defining property of the $W_{\pm}$ implies that $\tr AA^\top = 4 \det A A^\top$. This means that $AA^T = \lambda I$ for some $\lambda \ge 0$. Taking traces, we see that $\lambda = \mp \det A$. From this we easily get that $N \in W_{\pm}$ if and only if $J^{-1} J^{-\top} N^\top = \mp N^{\text{adj}}/\det J$, which is a linear condition.

So we can coordinize $V$ as
\begin{equation}\label{coord}
N = \left(\begin{matrix} a_+ + a_- & b_+ + b_-\\ j_{1, +}(a_+, b_+) +   j_{1, -}(a_-, b_-) &  j_{2, +}(a_+, b_+) +   j_{2, -}(a_-, b_-) \end{matrix}\right)
\end{equation}
with $a_+, a_-, b_+, b_-\in \mathbb{Q}$. In this way we obtain a linear map describing the change of variables $(n_1, n_2, n_3, n_4) \mapsto (a_+, a_-, b_+, b_-)$ with determinant $-1/4$. 

In these new coordinates we have
\begin{align} \label{short}
\| NJ^{-1} \| \pm 2\det NJ^{-1}  = (-b_{\mp}, a_{\mp}) \frac{J^{\top} J}{(\det J)^2} \left(\begin{matrix}-b_{\mp}\\ a_{\mp}\end{matrix}\right).
\end{align}
In particular, the condition $\| NJ^{-1} \|^2 \pm 2 \det NJ^{-1}  \preccurlyeq_{n} 1+  Z^{-2}$ is equivalent to $a_{\mp}, b_{\mp} \preccurlyeq_{n} 1+ 1/Z$. 

Depending on the sign of $\det N$, we will have two short variables and two long variables. For instance, when $\det N<0$, we have $\| NJ^{-1} \| - 2 |\det NJ^{-1}| = \| NJ^{-1} \| + 2\det NJ^{-1}$ and we obtain from \eqref{short} that $a_{-}, b_{-}  \preccurlyeq_{n} 1+ 1/Z$, which by \eqref{Z} is usually much shorter than the remaining two variables $a_{+}, b_{+}\preccurlyeq  k$. This is a reflection of the fact that $N$ is close to the two-dimensional subspace $W_+$. 
When $\det N>0$, the long and short variables will be interchanged. 

For later purposes we compute
\begin{equation}\label{diff}
\nabla_{a_+, b_+}  \| N J^{-1}\|^2 = 4J^{-1} J^{-\top}  \left(\begin{matrix} a_+\\ b_+ \end{matrix}\right)  , \quad \nabla_{a_+, b_+}  \det N J^{-1}  = -2J^{-1} J^{-\top}  \left(\begin{matrix} a_+\\ b_+ \end{matrix}\right). 
\end{equation}

Let us return to the notation $d = \det J$.  For $S \subseteq \mathbb{Q}^4$, we write $V[S] = \{N \in V: (a_+, a_-, b_+, b_-) \in S\}$ with respect to the   coordinates \eqref{coord}. Then clearly one has $M_2(\mathbb{Z}) \subseteq V[(\frac{1}{2d} \mathbb{Z})^4]$, and for volume reasons  $M_2(\mathbb{Z})$ is a sub-lattice of index $4d^4$ in $V[(\frac{1}{2d} \mathbb{Z})^4]$. Consequently, given $N_0 \in M_2(\mathbb{Z})$, there exist a vector $A_0 \in (\frac{1}{2d }\mathbb{Z})^4$ and a lattice $\Lambda \subseteq (\frac{1}{2}\mathbb{Z})^4$ of index $4d^4 \preccurlyeq 1$ such that
$$\{N \in M_2(\mathbb{Z}) : N \equiv N_0 \, (\text{mod } d)\} =  V[A_0 + \Lambda].$$
Note that the left hand side describes the summation condition in \eqref{prelim}. Recalling the present form of the $z$-integral \eqref{zint}, the term we have to bound is
\begin{align*}
    \frac{1}{k^2}&  \sum_{\substack{N \in  V[A_0 + \Lambda]\\ a_{\mp}, b_{\mp} \preccurlyeq_{n} 1+ 1/Z}} W_1\Big(\frac{\| N \|}{K_1}\Big) W_2\Big(\frac{\det N}{K_2^2}\Big)     \frac{1}{Z^{1/2} k} \int_0^{\infty} W_4\Big(\frac{z}{Z}; N\Big) E(z) \\
 &\times e\Big[\| N J^{-1} \|^2z - 2 \Big(\sqrt{(z \det NJ^{-1})^2 - \ell_0^2}  -  \ell_0 \arctan\Big(\frac{1}{\ell_0}\sqrt{ (z\det NJ^{-1} )^2 - \ell_0^2}\Big) \Big) \Big]\frac{dz}{z}
 \end{align*}
using the coordinates \eqref{coord} for $N$. At this point we can estimate trivially (by \eqref{Egen})
$$\preccurlyeq_{n} \frac{1}{k^2} k^2 \max\Big(1,  \frac{1}{Z^2}\Big) \frac{1}{Z^{1/2}k}  \min(Z^{1/2}, 1) = \frac{1}{\min(Z^{1/2}, Z^2) k}.$$ 
In particular, for the proof of Theorem \ref{thm1} we may assume that 
\begin{equation}\label{1/3}
   Z \ll k^{-1/3},
 \end{equation}  
    say, and hence $n= 2$ in \eqref{24}.

\subsubsection{Poisson summation} For notational simplicity let us only consider the case $a_-, b_- \preccurlyeq_{2} 1+ 1/Z$, the opposite case being completely identical. Recall this means that $\det NJ^{-1} < 0$. We now apply Poisson summation in the long variables $a_+, b_+$, but not in the short variables $a_-, b_-$. This was the whole point of introducing these coordinates. In this way it suffices to bound
\begin{align*}
\frac{1}{k^2}&  \sum_{\substack{ a_{-}, b_{-}\in \frac{2}{d}\mathbb{Z}\\ a_-, b_- \preccurlyeq_{2} 1+ 1/Z }} \sum_{\alpha, \beta \in \frac{1}{4d^4} \mathbb{Z}} \Big|\int_{\mathbb{R}}  \int_{\mathbb{R}} W_1\Big(\frac{\| N \|}{K_1}\Big) W_2\Big(\frac{\det N}{K_2^2}\Big)     \frac{1}{Z^{1/2} k} \int_0^{\infty} W_4\Big(\frac{z}{Z};N\Big) E(z) \\
 &\times e\Big[\| N J^{-1} \|^2z - 2 \Big(\sqrt{(z \det NJ^{-1})^2 - \ell_0^2}  -  \ell_0 \arctan\Big(\frac{1}{\ell_0}\sqrt{ (z\det NJ^{-1} )^2 - \ell_0^2}\Big) \Big) \Big]\frac{dz}{z} \\
 & \times e(\alpha  a_+ + \beta b_+) da_+\, db_+\Big|.
 \end{align*}
 Let,
\begin{align*}
\phi(a_+, b_+) 
=  &\| N J^{-1} \|^2z - 2 \Big(\sqrt{(z \det NJ^{-1})^2 - \ell_0^2}  -  \ell_0 \arctan\Big(\frac{1}{\ell_0}\sqrt{ (z\det NJ^{-1} )^2 - \ell_0^2}\Big) \Big) + \\
&+ \alpha  a_+ + \beta b_+
\end{align*}
be the phase. Recalling \eqref{diff},  we compute
\begin{align*}
\nabla_{a_+, b_+} \phi(a_+, b_+)&  = \left(\begin{matrix}\alpha\\ \beta\end{matrix}\right)   + z  \nabla_{a_+, b_+} \| N J^{-1} \|^2 - 2 \frac{\sqrt{(z \det NJ^{-1})^2 - \ell_0^2}}{\det NJ^{-1}}  \nabla_{a_+, b_+} \det NJ^{-1}\\
&= \left(\begin{matrix}\alpha\\ \beta\end{matrix}\right)  + 4z  \Big (1 -   \sqrt{ 1 - \frac{\ell_0^2}{z^2 (\det NJ^{-1})^2 }}    \Big)J^{-1} J^{-\top}  \left(\begin{matrix} a_+\\ b_+ \end{matrix}\right) \\
\end{align*}
and so 
\[ \nabla_{a_+, b_+} \phi(a_+, b_+) - \left(\begin{matrix}\alpha\\ \beta\end{matrix}\right)  \preccurlyeq \frac{1}{Zk}. \]
Moreover, 
%
 $$\frac{\partial^j}{\partial a^j_+} \phi(a_+, b_+) \preccurlyeq \frac{1}{k^j Z}$$
 for all fixed $j \geq 0$. 
 
  If $\alpha \not= 0$ (and hence $\alpha \succcurlyeq 1$), we apply Lemma \ref{integrationbyparts} (recall \eqref{flat} and \eqref{flatnew}) with
 $$R = |\alpha|, \quad  U \succcurlyeq_{3} k, \quad  Q \succcurlyeq k, \quad Y \preccurlyeq 1/Z$$
 to see that the $a_+$-integral is negligible.  Hence we may assume that $\alpha = 0$, and similarly we conclude $\beta = 0$. 

On the other hand, we have
$$
\Big \| J^{-1} J^{-\top}\left(\begin{matrix}    a_+ \\ b_+  \end{matrix}\right) \Big \| \approxeq \Big\| \left(\begin{matrix}    a_+ \\ b_+  \end{matrix}\right) \Big\|  \asymp  \| N \| \approxeq k  $$
by the support of the weight function $W_1$ and the fact that $a_-, b_- \preccurlyeq_{2} 1+ 1/Z$ with $Z$ satisfying \eqref{Z}. 
Thus in the case $\alpha = \beta = 0$ we have
$$\max\Big( \Big| \frac{\partial}{\partial a_+} \phi(a_+, b_+)\Big|,  \Big| \frac{\partial}{\partial b_+} \phi(a_+, b_+)\Big| \Big) \approxeq (kZ)^{-1} .$$
Once again we apply Lemma \ref{integrationbyparts} with
 $$R\succcurlyeq (kZ)^{-1}, \quad U \succcurlyeq_{3} k, \quad  Q \succcurlyeq k, \quad Y \preccurlyeq 1/Z$$
 to see by \eqref{1/3} that at least one of the $a_+, b_+$-integrals is negligible. 
 
 This completes the proof.

\end{document}